\documentclass[preprint,11pt]{elsarticle}
\usepackage{stmaryrd}
\usepackage{amsmath,amssymb,amsthm}
\usepackage{enumitem}
\usepackage{url}
\usepackage[colorlinks=true,linkcolor=blue,citecolor=blue,urlcolor=blue]{hyperref}

\newcommand{\R}{\mathbb R}

\def \Z{\Bbb Z}

\def \R{\Bbb R}

\def \wt{{\textit wt}}

\def \Res{{\rm Res}}
\def \End{{\rm End}}

\def \Id{{\rm Id}}
\def \Hom{{\rm Hom}}

\def \<{\langle}
\def \>{\rangle}

\def \1{{\bf 1}}

\def \({{\rm (}}
\def \){{\rm )}}

\def \1{{\bf 1}}

\def\Hom{{\rm Hom}}

\def\Res{{\rm Res}}

\newcommand{\boxtimesc}{\Box_V}
\newtheorem{Theorem}{Theorem}[section]
\newtheorem{Proposition}[Theorem]{Proposition}
\newtheorem{Lemma}[Theorem]{Lemma}

\newtheorem{Corollary}[Theorem]{Corollary}
\newtheorem{Remark}[Theorem]{Remark}
\newtheorem{Main Theorem}[Theorem]{Main Theorem}

\newtheorem{Definition}[Theorem]{Definition}
\begin{document}

\begin{frontmatter}

\title{Cotensor decompositions for comodules over graded vertex operator coalgebras}

\author[addr1]{Hao Wang\footnote{Email address: whaomath@nwu.edu.cn. Supported by a NSFC grant 12001426}}

\address[addr1]{School of Mathematics, Northwest University, Xi'an 710127, Shaanxi, China}

\begin{abstract}
Let $V$ be a graded vertex operator coalgebra and $\mathcal{C}$ the category of
admissible $V$-comodules. We first give the construction of contragredient comodule of an admissible
$V$-comodule.
We then introduce cointertwining operators among admissible $V$-comodules. Using cointertwining operators, we define the cotensor
decomposition of an admissible $V$-comodule by a universal property.
Finally, we prove that the complexified Grothendieck group
$
\mathbb{K}[\mathcal{C}]
=
\mathbb{C}\otimes_{\mathbb{Z}}K[\mathcal{C}]
$
is a cocommutative coassociative coalgebra. Moreover, this coalgebra is naturally dual to the fusion
algebra of the category of admissible modules over the dual vertex operator
algebra $V'$.
\end{abstract}


\end{frontmatter}

\section{Introduction}

Vertex operator algebras were introduced in connection with two-dimensional conformal
field theory \cite{BPZ} and the moonshine module, and have become fundamental algebraic structures
in representation theory, conformal field theory, and related areas \cite{LL}. The representation
theory of vertex operator algebras is deeply connected with the theory of tensor
categories. In particular, the theory of tensor products for modules over a vertex
operator algebra was developed by Huang and Lepowsky in a series of works
\cite{HL1,HL2,HL3,Hu3}, and plays an essential role in the study of fusion
rules and modular tensor categories \cite{Hu2,Hu1,Z}. For a rational vertex operator algebra, the
Grothendieck group of its module category carries a natural commutative associative
algebra structure, whose structure constants are given by the fusion rules.

Motivated by classical algebras and Lie algebras \cite{M,M1, M2,NT}, there is also a coalgebraic counterpart of the theory of vertex operator algebras.
Vertex coalgebras and vertex operator coalgebras were introduced and studied by
Hubbard \cite{H,H1,H2}. Under suitable
finite-dimensionality assumptions, the graded dual of a vertex operator algebra is a
graded vertex operator coalgebra, and vice versa. This duality suggests that many
constructions in the representation theory of vertex operator algebras should admit
coalgebraic counterparts for comodules over vertex operator coalgebras. In \cite{W,W1}, we proved the dual Zhu's theory which is the coalgebraic part of Zhu's theory for vertex operator algebras \cite{DLM1,DLM2,DLM3,Z}. Now we are intending to study the comodule category of a vertex operator coalgebra, and this will help us to work on vertex operator bialgebras \cite{HLX,JKLT,XH} and nonlocal vertex coalgebras \cite{L4,L2} in future.

The purpose of this paper is to develop a coalgebraic analogue of the tensor product
theory for admissible modules. More precisely, let $V$ be a graded vertex operator
coalgebra and let $\mathcal{C}$ be the category of admissible $V$-comodules. We
introduce cointertwining operators among admissible $V$-comodules. These operators are
dual to intertwining operators among modules over the dual vertex operator
algebra $V'$ \cite{FHL,HL1,HL2,HL3}. If
$$
\mathcal {W}(z):\mathcal {M}^3\longrightarrow \mathcal {M}^1\otimes \mathcal {M}^2\{z\}
$$
is a cointertwining operator of type
\[
\binom{\mathcal {M}^1\quad \mathcal {M}^2}{\mathcal {M}^3},
\] then it should be regarded as a coalgebraic analogue of
an intertwining operator of type
$
\binom{(\mathcal {M}^3)'}{(\mathcal {M}^1)'\quad (\mathcal {M}^2)'}
$
for the vertex operator algebra $V'$.

Before introducing cointertwining operators, we study the contragredient comodule of an
admissible $V$-comodule. For an admissible $V$-comodule $\mathcal {M}$ with finite-dimensional
homogeneous subspaces, we construct a natural admissible $V$-comodule structure on its
graded dual $\mathcal {M}'$. This construction is dual to the contragredient module construction
of Frenkel, Huang, and Lepowsky \cite{FHL}. We also prove that $\mathcal {M}''$ is naturally
isomorphic to $\mathcal {M}$ as an admissible $V$-comodule.

Using cointertwining operators, we define the cotensor decomposition of an admissible
$V$-comodule by a universal property. This is the coalgebraic analogue of the tensor
product of modules for vertex operator algebras. In the module case, the tensor product
of two modules represents intertwining operators with fixed source modules. In the
coalgebraic setting, the direction of the operator is reversed. Thus an admissible
comodule $\mathcal {M}$ is decomposed into pairs of admissible comodules through cointertwining
operators.
If $V$ is simple and corational, then by the dual Zhu's theory developed in \cite{W},
the category $\mathcal{C}$ is semisimple with finitely many isomorphism classes of
simple admissible comodules. Let
$
\{\mathcal {M}^i\}_{i\in I}
$
be a complete set of representatives of simple admissible $V$-comodules. We prove that
the cotensor decomposition of $\mathcal {M}$ has the form
\[
\Delta_{\boxtimesc}([\mathcal {M}])
=
\sum_{i,j\in I}
N^{\mathcal {M}^i,\mathcal {M}^j}_{\mathcal {M}}
[\mathcal {M}^i]\boxtimes [\mathcal {M}^j],
\]
where the coefficient is the corresponding cointertwining fusion rule.

One of the main results of this paper is that the cotensor decomposition induces a
cocommutative and coassociative coproduct on the complexified Grothendieck group
\[
\mathbb{K}[\mathcal{C}]
=
\mathbb{C}\otimes_{\mathbb{Z}}K[\mathcal{C}].
\]
The cocommutativity follows from the skew-symmetry of cointertwining operators, while
the coassociativity follows from the associativity of tensor products for admissible
modules over the rational vertex operator algebra $V'$ \cite{Hu3}. More precisely, we establish a
duality between cointertwining fusion rules for $V$-comodules and ordinary fusion rules
for $V'$-modules:
\[
N^{\mathcal {M}^i,\mathcal {M}^j}_{\mathcal {M}^k}
=
N^{(\mathcal {M}^k)'}_{(\mathcal {M}^i)',(\mathcal {M}^j)'}.
\]
This identity allows us to transfer the associativity of the fusion algebra of $V'$ to
the coassociativity of the Grothendieck coalgebra of $\mathcal{C}$.

Finally, we compare the Grothendieck coalgebra of admissible $V$-comodules with the
fusion algebra of admissible modules over $V'$. We show that the coalgebra
$
\mathbb{K}[\mathcal{C}]
$
is naturally dual to the fusion algebra
$
\mathbb{K}[\mathcal{C}^{op}]
$
of the corresponding category of admissible $V'$-modules. Thus the cotensor theory
developed here provides a coalgebraic dual of the usual fusion theory for rational
vertex operator algebras \cite{Hu3,Hu2,Hu1}.

The paper is organized as follows. In Section 2, we recall basic definitions and facts
about algebras, coalgebras, vertex operator algebras, vertex operator coalgebras, and
admissible comodules. In Section 3, we construct the contragredient comodule of an
admissible $V$-comodule. In Section 4, we introduce cointertwining operators and prove
their basic symmetry and duality properties. In Section 5, we define cotensor
decompositions and prove that the complexified Grothendieck group of admissible
$V$-comodules is a cocommutative coassociative coalgebra. In Section 6, we relate this
coalgebra to the fusion algebra of the graded dual vertex operator algebra $V'$.

In this paper, we work over complex field $\mathbb{C},$ all vector spaces, linear maps will be over $\mathbb{C},$ $\otimes$ means $\otimes _\mathbb{C}$. $\Res_zf(z)$ is the coefficient of $z^{-1}$. Furthermore, we have $$\Res_z\frac{d}{dz}f(z)\cdot g(z)=-\Res_zf(z)\cdot\frac{d}{dz}g(z).$$
For any vector spaces $V,W$, the flipping map $\tau:V\otimes W\rightarrow W\otimes V$ is defined by $$\tau(v\otimes w)=w\otimes v,$$ for any $v\in V,w\in W.$ In this paper, we always view $V\otimes \mathbb{C}=\mathbb{C}\otimes V=V$ naturally. For $n\in\mathbb{C}$, $(z_1+z_2)^n=\sum_{i\geq0}\tbinom{n}{i}z_1^{n-i}z_2^i.$ Let $V$ be a vector space, denote $$V^*=\Hom(V,\mathbb{C})$$ by the dual space of $V$. Let $(\cdot,\cdot)$ be the natural pair between $V^*$ and $V$.

\section{Preliminaries}
In this section, we recall several definitions and results about algebras, coalgebras, vertex operator algebras and vertex operator coalgebras.

\subsection{Algebras and coalgebras}
\begin{Definition}\cite{DNR}\label{DA2.1}
  An associative algebra is a triple $(A,\mu,\eta)$, where $A$ is a vector space, $\mu: A\otimes A\rightarrow A$ and $\eta:\mathbb{C}\rightarrow A$ are linear maps such that:

  (i) For any $a\in A$, we have $\mu(a\otimes\eta(1))=\mu(\eta(1)\otimes a)=a$. They are called unit identities.

  (ii) For any $a,b,c\in A$, the following identity $$\mu(a\otimes\mu(b\otimes c))=\mu(\mu(a\otimes b)\otimes c)$$ holds. This identity is called associativity. We write $\mu(a\otimes b)=ab$.
\end{Definition}

\begin{Definition}\cite{DNR}\label{DC2.1}
  A coassociative coalgebra is a triple $(C,\Delta,\epsilon)$, where $C$ is a vector space, $\Delta:C\rightarrow C\otimes C$ and $\epsilon:C\rightarrow \mathbb{C}$ are linear maps such that:

  (i) For any $a\in C$, the following identities $$(\Id_C\otimes \epsilon)\circ \Delta(a)=a=(\epsilon\otimes \Id_C)\circ \Delta(a)$$ hold. They are called counit identities.

  (ii) For any $a\in C$, the following identity $$(\Id_C\otimes \Delta)\circ \Delta(a)=(\Delta\otimes \Id_C)\circ \Delta(a)$$ holds. This identity is called coassociativity.
\end{Definition}

\begin{Definition}\cite{DNR}
  A coassociative coalgebra $(C,\Delta,\epsilon)$ is called cocommutative if $$\tau\circ\Delta=\Delta.$$
\end{Definition}

\begin{Proposition}\cite{DNR,K}\label{A-C}
Let $(A,\mu,\eta)$ be an algebra and $(C,\Delta,\epsilon)$ a coalgebra. Then, we have

(i) $C^*$ is an algebra with $$(\mu_{C^*}(f\otimes g),a)=\sum f(a')g(a''),\eta_{C^*}(1)=\epsilon,$$ where $f,g\in C^*, a\in C$.

(ii) $A^\circ$ is a coalgebra with $$(\Delta_{A^\circ}(f),a\otimes b)=f(ab),\epsilon_{A^\circ}(f)=f(\eta(1)),$$ where $f\in A^\circ, a,b\in A.$ If $\dim A<\infty,$ then $A^\circ=A^*.$
\end{Proposition}

\subsection{Vertex operator algebras}
 \begin{Definition}\cite{LL}\label{Def2.1}
  Let $\mathcal {V}=\oplus_{s\in \mathbb{Z}}\mathcal {V}_s$ be a $\mathbb{Z}$-graded vector space with $\mathcal {V}_s=0$, $0\gg s$, $\dim \mathcal {V}_s<\infty,\forall s\in \mathbb{Z}$, and $\textbf{1}\in \mathcal {V}_0,\omega\in \mathcal {V}_2,Y(\cdot,z):\mathcal {V}\otimes \mathcal {V}\rightarrow \End(\mathcal {V})[[z,z^{-1}]],u\otimes v\mapsto Y(u,z)v=\sum_{t\in\mathbb{Z}}u_tvz^{-t-1}$, where $u_t\in \End(\mathcal {V})$. Then, $(\mathcal {V},Y,\textbf{1},\omega)$ is called a vertex operator algebra if the following hold:

  (i) For $ u,v\in \mathcal {V},$ $u_tv=0$, if $t\gg0$.

  (ii) $Y(\textbf{1},z)v=v,\lim_{z\rightarrow0}Y(v,z)\textbf{1}=v$, for $ v\in \mathcal {V}$.

  (iii) Write $Y(\omega,z)=\sum_{t\in\mathbb{Z}}\omega_tz^{-t-1}=\sum_{t\in\mathbb{Z}}L_\mathcal {V}(t)z^{-t-2}$, then
  \begin{eqnarray*}
  &&\mathcal {V}_s=\{v\in \mathcal {V}|L_\mathcal {V}(0)v=sv\},~Y(L_\mathcal {V}(-1)v,z)=\frac{d}{dz}Y(v,z),\\
  &&[L_\mathcal {V}(p),L_\mathcal {V}(q)]=(p-q)L_\mathcal {V}(p+q)+\delta_{p+q,0}\frac{p^3-p}{12}d,
  \end{eqnarray*}
  where $d\in\mathbb{C}$ is called central charge of $\mathcal {V}$. For $v\in \mathcal {V}_s$, $v$ is said to be homogeneous and the weight $\wt v$ of $v$ is defined to be $s$.

  (iv) For $ u,v\in \mathcal {V}$, we have
  \begin{eqnarray*}
  & &z_0^{-1}\delta(\frac{z_1-z_2}{z_0})Y(u,z_1)Y(v,z_2)-z_0^{-1}\delta(\frac{-z_2+z_1}{z_0})Y(v,z_2)Y(u,z_1)\\
  & &\ \ \ \ \ =z_1^{-1}\delta(\frac{z_2+z_0}{z_1})Y(Y(u,z_0)v,z_2).
  \end{eqnarray*}
  \end{Definition}

\begin{Definition}\cite{LL}\label{Def2.2}
  Let $\mathcal {V}$ be a vertex operator algebra. A {\em weak  $\mathcal {V}$-module} $\mathcal {N}$ is a vector space equipped
with a linear map
\begin{align*}
Y_{\mathcal {N}}:\mathcal {V}&\to (\End \mathcal {N})[[z, z^{-1}]],\\
v&\mapsto Y_{\mathcal {N}}(v,z)=\sum_{s\in\Z}v_sz^{-s-1},\,v_s\in \End \mathcal {N},
\end{align*}
satisfying the following conditions: For any $u\in \mathcal {V},\ v\in \mathcal {V},\ w\in \mathcal {N}$ and $t\in \Z$,
\begin{align*}
&\ \ \ \ \ \ \ \ \ \ \ \ \ \ \ \ \ \ \ \ \ \ \ \ \ \ u_tw=0 \text{ for } t\gg0;\\
&\ \ \ \ \ \ \ \ \ \ \ \ \ \ \ \ \ \ \ \ \ \ \ \ \ \ Y_\mathcal {N}(\1, x)=\Id_\mathcal {N};\\
\begin{split}
&x_{0}^{-1}\delta\left(\frac{z_{1}-z_{2}}{z_{0}}\right)Y_{\mathcal {N}}(u,z_{1})Y_\mathcal {N}(v,z_{2})-z_{0}^{-1}\delta\left(
\frac{z_{2}-z_{1}}{-z_{0}}\right)Y_\mathcal {N}(v,z_{2})Y_\mathcal {N}(u,z_{1})\\
&\quad=x_{2}^{-1}\delta\left(\frac{z_{1}-z_{0}}{z_{2}}\right)Y_\mathcal {N}(Y(u,z_{0})v,z_{2}).
\end{split}
\end{align*}

A weak
 $\mathcal {V}$-module  $\mathcal {N}$ is called an \textit{admissible $\mathcal {V}$-module} if $\mathcal {N}$ has a $\Z_{\geq
0}$-gradation $\mathcal {N}=\bigoplus_{s\in\Z_{\geq 0}}N(s)$ such
that
\begin{align*}\label{AD1}
v_tN(s)\subset N(\wt{v}+s-t-1)
\end{align*}
for any homogeneous $v\in \mathcal {V}$ and $s,\,t\in\Z$.

An admissible $\mathcal {V}$-module $\mathcal {N}$ is said to be
\textit{irreducible} if $\mathcal {N}$ has no non-trivial admissible
$\mathcal {V}$-submodule. When an admissible $\mathcal {V}$-module $\mathcal {N}$ is
direct sum of irreducible admissible submodules, $\mathcal {N}$ is called
\textit{completely reducible}.

$\mathcal {V}$ is called rational if every admissible $\mathcal {V}$-module is completely reducible.
\end{Definition}

\begin{Definition}\cite{FHL}\label{Def2.3}
Let $(\mathcal {N},Y_\mathcal {N})$ be an admissible $\mathcal {V}$-module. The contragredient $\mathcal {V}$-module $(\mathcal {N}',Y'_{\mathcal {N}'})$ is defined as follows:
\begin{eqnarray*}
&&\mathcal {N}'=\bigoplus_{s\in\Z_{\geq 0}}\mathcal {N}(s)^*,\\
&&(Y'_{\mathcal {N}'}(v,z)n^*,n)=(n^*,Y_\mathcal {N}(e^{zL(1)}(-z^{-2})^{L(0)}v,z^{-1})n),
\end{eqnarray*}
for any $n^*\in \mathcal {N}',v\in \mathcal {V},n\in \mathcal {N}$. Furthermore, if all homogeneous subspaces of $\mathcal {N}$ are finite dimensional, then $\mathcal {N}''\cong\mathcal {N}$ as admissible $\mathcal {V}$-modules.
\end{Definition}

\begin{Definition}\cite{FHL}\label{Def2.4}
  Let $\mathcal {N}_1,\mathcal {N}_2,\mathcal {N}_3$ be three weak $\mathcal {V}$-modules. An intertwining operator $\mathcal {Y}(\cdot,z)$ of type $\tbinom{\mathcal {N}_3}{\mathcal {N}_1~~\mathcal {N}_2}$ is a linear map
   \begin{eqnarray*}
  \mathcal {Y}(\cdot,z):~\mathcal {N}_1&\rightarrow&~\Hom(\mathcal {N}_2,\mathcal {N}_3)\{z\},\\
  n_1&\mapsto& \mathcal {Y}(n_1,z)=\sum_{t\in\mathbb{C}}n_{1,t}z^{-t-1},
  \end{eqnarray*}
  where $n_{1,t}\in \Hom(\mathcal {N}_2,\mathcal {N}_3)$, satisfying the following conditions:

  (i) For any $n_1\in \mathcal {N}_1,n_2\in \mathcal {N}_2,\lambda\in\mathbb{C}$, $n_{1,t+\lambda}n_2=0$ for $t\in\mathbb{Z}$ sufficiently large.

  (ii) For any $n_1\in \mathcal {N}_1$, $\frac{d}{dz}\mathcal {Y}(n_1,z)=\mathcal {Y}(L(-1)n_1,z)$.

  (iii) For any $v\in V, n_1\in \mathcal {N}_1$,
  \begin{eqnarray*}
  z_0^{-1}\delta(\frac{z_1-z_2}{z_0})Y_{\mathcal {N}_3}(v,z_1)\mathcal {Y}(n_1,z_2)-z_0^{-1}\delta(\frac{-z_2+z_1}{z_0})\mathcal {Y}(n_1,z_2)Y_{\mathcal {N}_2}(v,z_1)\\
  =z_2^{-1}\delta(\frac{z_1-z_0}{z_2})\mathcal {Y}(Y_{\mathcal {N}_1}(v,z_0)n_1,z_2),
  \end{eqnarray*}
  \end{Definition}

  Obviously, all the intertwining operators of type $\tbinom{\mathcal {N}_3}{\mathcal {N}_1~~\mathcal {N}_2}$ form a vector space, denoted by $I_\mathcal {V}\tbinom{\mathcal {N}_3}{\mathcal {N}_1~~\mathcal {N}_2}.$
  The numbers $N_{\mathcal {N}_1,\mathcal {N}_2}^{\mathcal {N}_3}=\dim I_\mathcal {V}\tbinom{\mathcal {N}_3}{\mathcal {N}_1~~\mathcal {N}_2}$
  are called the fusion rules.

\begin{Proposition}\cite{HL1,HL2,HL3,L3}\label{Prop2.5}
Let $\mathcal {V}$ be a rational vertex operator algebra and $\mathcal {N}_1,\mathcal {N}_2,\mathcal {N}_3$ be three admissible $\mathcal {V}$-modules. Then we have $$N_{\mathcal {N}_1,\mathcal {N}_2}^{\mathcal {N}_3}=N_{\mathcal {N}_2,\mathcal {N}_1}^{\mathcal {N}_3},~N_{\mathcal {N}_1,\mathcal {N}_2}^{\mathcal {N}_3}=N_{\mathcal {N}_1,\mathcal {N}'_3}^{\mathcal {N}'_2}.$$
\end{Proposition}

\begin{Definition}\cite{HL1,HL2,HL3}\label{Def2.6}
Let $\mathcal {N}_1,\mathcal {N}_2$ be two admissible $\mathcal {V}$-modules. An admissible $\mathcal {V}$-module $(\mathcal {N},I)$ is called a tensor product of $\mathcal {N}_1$ and $\mathcal {N}_2$, where $I\in I_\mathcal {V}\tbinom{\mathcal {N}}{\mathcal {N}_1~~\mathcal {N}_2}$, if for any admissible $\mathcal {V}$-module $\mathcal {N}_3$ and any intertwining operator $\mathcal {Y}\in I_\mathcal {V}\tbinom{\mathcal {N}_3}{\mathcal {N}_1~~\mathcal {N}_2}$, there is a unique $\mathcal {V}$-module homomorphism $\varphi:\mathcal {N}\rightarrow \mathcal {N}_3$, such that $\mathcal {Y}=\varphi\circ I$. We denote $(\mathcal {N},I)$ by $\mathcal {N}_1\boxtimes_\mathcal {V}\mathcal {N}_2$.
\end{Definition}

\begin{Proposition}\cite{HL1,HL2,HL3,Hu3,L3}\label{Prop2.7}
Let $\mathcal {V}$ be a rational vertex operator algebra. The tensor product $\boxtimes_\mathcal {V}$ is commutative and associative.
\end{Proposition}

\subsection{Vertex operator coalgebras}

\begin{Definition}\cite{H,H2}\label{Def2.8}
  A vertex coalgebra is a triple $(V,\Yup (z),c)$, where $V$ is a vector space, $c: V\rightarrow\mathbb{C}$ is a linear map, $\Yup (z):V\rightarrow V\otimes V[[z,z^{-1}]]$, $v\mapsto \sum_{k\in\mathbb{Z}}\Delta_k(v)z^{-k-1}$ is a linear map, which satisfy following conditions:

  (i) For any $v\in V$, we have $\Delta_k(v)$ is a finite sum and $\Delta_k(v)=0$ for $k<<0.$

  (ii) For any $v\in V$, we have
    \begin{eqnarray}\label{LU3.1}
      (c\otimes \Id)\circ\Yup (z)v=v,
    \end{eqnarray} and
    \begin{eqnarray}\label{CI3.2}
      (\Id\otimes c)\circ\Yup (z)v\in V[[z]], \label{CI3.2}\\
      \lim_{z\rightarrow 0}(\Id\otimes c)\circ\Yup (z)v=v.\label{CI3.3}
    \end{eqnarray}
$c$ is called the counit.

  (iii) The following identity
    \begin{eqnarray}\label{JI3.4}
      &&z_0^{-1}\delta(\frac{z_1-z_2}{z_0})(\Id\otimes \Yup (z_2))\circ \Yup (z_1)-z_0^{-1}\delta(\frac{z_2-z_1}{-z_0})(\tau\otimes\Id)\circ(\Id\otimes \Yup (z_1))\circ \Yup (z_2)\nonumber\\
      &&\ \ \ \ \ =z_1^{-1}\delta(\frac{z_2+z_0}{z_1})(\Yup (z_0)\otimes \Id)\circ \Yup (z_2)
    \end{eqnarray}
  holds on $V$. This identity is also called Jacobi identity.

  $A$ vertex coalgebra $(V,\Yup(z),c)$ is called graded if there is a $\mathbb{Z}$-gradation $V=\oplus_{s\in\mathbb{Z}}V_s$ on $V$ such that each homogeneous space is finite dimensional and $V_s=0$ for $s<<0.$ View $V\otimes V$ as a $\mathbb{Z}$-graded vector space with natural gradation, i.e., $$(V\otimes V)_t=\oplus_{s\in\mathbb{Z}}V_{t-s}\otimes V_{s}.$$ $\Delta_k$ is a homogeneous linear map of degree $k+1$, i.e., for any $v\in V_s$, we have $$\Delta_k(v)\in (V\otimes V)_{s+k+1}.$$

  A $\mathbb{Z}$-graded vertex coalgebra is called a $\mathbb{Z}$-graded vertex operator coalgebra, if there is another linear map $\rho:V\rightarrow \mathbb{C}$, and write $(\rho\otimes \Id)\circ\Yup (z)=\sum_{k\in\mathbb{Z}}L_V(k)z^{k-2}$, then $L_V(k)\in\End(V)$ and the following Virasoro identity
    \begin{eqnarray}\label{VI3.5}
     [L_V(k),L_V(j)]=(k-j)L_V(k+j)+\frac{k^3-k}{12}\delta_{k+j,0}d
    \end{eqnarray}
  holds on $V$. $d$ is called the rank of $V$. Furthermore, $V_s$ is an eigenspace of $L_V(0)$ with eigenvalue $s$ and the following identity
    \begin{eqnarray}\label{L(1D)3.6}
     (L_V(1)\otimes\Id)\circ \Yup (z)=\frac{d}{dz}\Yup (z)=\Yup(z)\circ L_V(1)-\Id\otimes L_V(1)\circ \Yup(z)
    \end{eqnarray}
  holds on $V$, called $L_V(1)$-derivation identity. For any $v\in V_s$, we say $v$ is homogeneous of weight $s$, denoted by $s=\wt v.$

  If there is no confusion, we may say $V$ is a graded vertex operator coalgebra for simplicity.
\end{Definition}

\begin{Lemma}\label{Lm2.9}\cite{H2}
  Let $(V,\Yup (z),c)$ be a graded vertex coalgebra. Then $$\tau\circ\Yup(z)=\Yup(-z)\circ\exp^{zD^*},$$ where $D^*=\Res_zz^{-2}(\Id\otimes c)\circ\Yup(z)$. It is obvious that $D^*$ is homogeneous of degree $-1$. This is called skew-symmetry.

  Furthermore, if $(V,\Yup (z),c,\rho)$ is a graded vertex operator coalgebra, we have$$\tau\circ\Yup(z)=\Yup(-z)\circ\exp^{zL_V(1)}.$$
\end{Lemma}

\begin{Lemma}\label{Lm2.10}\cite{H2}
  Let $(V,\Yup (z),\rho,c)$ be a graded vertex operator coalgebra. Then we have
      \begin{eqnarray*}
      \Id\otimes e^{-zL_V(1)}\circ\Yup(z_1)\circ e^{zL_V(1)}&=&\Yup(z_1+z),\\
      e^{zL_V(1)}\otimes \Id\circ\Yup(z_1)&=&\Yup(z_1+z).
    \end{eqnarray*}
\end{Lemma}

\begin{Definition}\cite{H,H2}\label{Def2.10}
   Let $(V,\Yup (z),c,\rho)$ be a graded vertex operator coalgebra. Let $\mathcal {M}$ be a vector space and $\Yup_\mathcal {M}(z):\mathcal {M}\rightarrow V\otimes \mathcal {M}[[z,z^{-1}]]$ a linear map. $(\mathcal {M},\Yup_\mathcal {M}(z))$ is called a weak $V$-comodule, if

  (i) For any $m\in \mathcal {M}$, write $\Yup_\mathcal {M}(z)=\sum_{k\in\mathbb{Z}}\Delta_{\mathcal {M},k}z^{-k-1}$, we have $\Delta_{\mathcal {M},k}(m)$ is a finite sum and $\Delta_{\mathcal {M},k}(m)=0$ for $k<<0.$

  (ii) For any $m\in \mathcal {M}$, we have
    \begin{eqnarray}\label{LUM3.1}
      (c\otimes \Id )\circ\Yup_\mathcal {M} (z)m=m,
    \end{eqnarray}
called left counit identity.

  (iii) The following identity
    \begin{eqnarray}\label{JIM3.4}
      &&z_0^{-1}\delta(\frac{z_1-z_2}{z_0})(\Id\otimes \Yup_\mathcal {M} (z_2))\circ \Yup_\mathcal {M} (z_1)\nonumber\\
      &&~~~~~~~~~~-z_0^{-1}\delta(\frac{z_2-z_1}{-z_0})(\tau\otimes\Id )\circ(\Id\otimes \Yup_\mathcal {M} (z_1))\circ \Yup_\mathcal {M} (z_2)\nonumber\\
      &&=z_1^{-1}\delta(\frac{z_2+z_0}{z_1})(\Yup (z_0)\otimes \Id )\circ \Yup_\mathcal {M} (z_2)
    \end{eqnarray}
  holds on $\mathcal {M}$. This identity is also called Jacobi identity.

  (iv) Write $(\rho\otimes \Id )\circ\Yup_\mathcal {M} (z)=\sum_{k\in\mathbb{Z}}L_V(k)z^{k-2}$, then $L_V(k)\in\End(\mathcal {M})$ and the following Virasoro identity
    \begin{eqnarray}\label{VIM3.5}
     [L_V(k),L_V(j)]=(k-j)L_V(k+j)+\frac{k^3-k}{12}\delta_{k+j,0}d
    \end{eqnarray}
  holds on $\mathcal {M}$.

  (v) The following identities
    \begin{eqnarray}\label{L(1D)M3.6}
     (L_V(1)\otimes\Id )\circ \Yup_\mathcal {M} (z)=\frac{d}{dz}\Yup_\mathcal {M} (z)=\Yup_\mathcal {M}(z)\circ L_V(1)-(\Id\otimes L_V(1))\circ\Yup_\mathcal {M}(z)
    \end{eqnarray}
  holds on $\mathcal {M}$, called $L_V(1)$-derivation property.

  $(\mathcal {M},\Yup_\mathcal {M}(z))$ is called an admissible $V$-comodule, if there is a $\mathbb{N}$-grading on $\mathcal {M}$, such that $\mathcal {M}=\oplus_{t\in\mathbb{N}}M(t)$, and $\Delta_k$ is a homogeneous linear map of degree $k+1$, i.e., for any $m\in M_t$, we have $$\Delta_{\mathcal {M},k}(m)\in (V\otimes \mathcal {M})_{t+k+1}=\oplus_{s\in\mathbb{N}}V_{t+k+1-s}\otimes M(s).$$

  Let $(\mathcal {M},\Yup_\mathcal {M}(z))$ be an admissible $V$-comodule, $\mathcal {M}_1$ a subspace of $\mathcal {M}$. If $(\mathcal {M}_1,\Yup_\mathcal {M}|_{\mathcal {M}_1}(z))$ is also an admissible $V$-comodule, we say $\mathcal {M}_1$ is an admissible sub comodule.

  Let $(\mathcal {M},\Yup_\mathcal {M}(z))$ be an admissible $V$-comodule. $\mathcal {M}$ is called simple if there is no nontrivial admissible sub comodule. $\mathcal {M}$ is called cosemisimple if $\mathcal {M}$ is a direct sum of simple admissible sub comodules.

  Let $(\mathcal {M}_1,\Yup_{\mathcal {M}_1}(z))$ and $(\mathcal {M}_2,\Yup_{\mathcal {M}_2}(z))$ be two admissible $V$-comodules. A linear map $\psi:\mathcal {M}_1\rightarrow \mathcal {M}_2$ is called a $V$-comodule homomorphism if $$(\Id\otimes \psi)\circ \Yup_{\mathcal {M}_1}(z)=\Yup_{\mathcal {M}_2}(z)\circ \psi.$$

  Similarly, we have the definitions of monomorphism, epimorphism and isomorphism, etc.

  A graded vertex operator coalgebra $V$ is called corational if every admissible $V$-comodule is cosemisimple.
\end{Definition}

\begin{Remark}\label{Rmk2.11}
(i) $L_V(k)$ is a homogeneous linear map with degree $-k.$

(ii) We use $L_V(k)$ to denote linear maps both on $V$ and admissible $V$-comodule $\mathcal {M}$.

(iii) We use $L_\mathcal {V}(k)$ to denote linear maps both on $\mathcal {V}$ and admissible $\mathcal {V}$-module $\mathcal {N}$.
\end{Remark}

\begin{Remark}\label{Rmk3.22}
  Let $(V,\Yup (z),c,\rho)$ be a graded vertex operator coalgebra, and $\mathcal {M}$ is an admissible $V$-comodule and $m\in M_t$, from definition, we know $\Delta_{\mathcal {M},k}(m)\in\oplus_{s\in\mathbb{N}}V_{t+k+1-s}\otimes M_s$. Hence we can write $$\Delta_{\mathcal {M},k}(m)=\sum\sum_{s\in\mathbb{N}}m'_{t+k+1-s}\otimes m''_{s}$$with $m_i'\in V_i,m''_i\in M_i$.
\end{Remark}

\begin{Proposition}\cite{H2}\label{WCM}
Let $V$ be a graded vertex operator coalgebra. In the definition of weak $V$-comodule, Jocabi identity (\ref{JIM3.4}) is equivalent to following two conditions:

(i) (Weak cocommutativity) For $m\in \mathcal {M}$, there exists $q\in \mathbb{N}$ such that $$(z_1-z_2)^q(\Id_V\otimes \Yup_\mathcal {M} (z_2))\circ \Yup_\mathcal {M} (z_1)m=(z_1-z_2)^q(\tau\otimes\Id_\mathcal {M})\circ(\Id_V\otimes \Yup_\mathcal {M} (z_1))\circ \Yup_\mathcal {M} (z_2)m.$$

(ii) (Weak coassociativity) For $m\in \mathcal {M}$, there exists $q\in \mathbb{N}$ such that $$(z_0+z_2)^q(\Yup (z_0)\otimes \Id_\mathcal {M})\circ \Yup_\mathcal {M} (z_2)m=(z_0+z_2)^q(\Id_V\otimes \Yup_\mathcal {M} (z_2))\circ \Yup_\mathcal {M} (z_0+z_2)m.$$
\end{Proposition}

\begin{Proposition}\label{Prop2.15}
Let $V$ be a graded vertex operator coalgebra, and $(\mathcal {M},\Yup_\mathcal {M}(z))$ an admissible $V$-comodule. Then, we have $$e^{z_0L_V(1)}\otimes\Id\circ\Yup_\mathcal {M}(z)=\Yup_\mathcal {M}(z+z_0)=\Id\otimes e^{-z_0L_V(1)}\circ\Yup_\mathcal {M}(z)\circ e^{z_0L_V(1)}.$$
\end{Proposition}
\proof This follows from $L_V(1)$-derivation property immediately.                             $\hfill\Box$

\begin{Proposition}\cite{H,H1,H2}\label{Prop2.12}
Let $(V,\Yup (z),c,\rho)$ be a graded vertex operator coalgebra and $(\mathcal {V},Y,\textbf{1},\omega)$ a vertex operator algebra. Let $(\mathcal {M},\Yup_\mathcal {M}(z))$ be an admissible $V$-comodule and $(\mathcal {N},Y_\mathcal {N})$ an admissible $\mathcal {V}$-module. Suppose $V_0\nsubseteq\ker c,V_2\nsubseteq\ker\rho.$ Then we have

(i) $(\mathcal {V}'=\oplus_{s\in\mathbb{Z}}\mathcal {V}_s^*,\Yup_{\mathcal {V}'}(z),c_{\mathcal {V}'},\rho_{\mathcal {V}'})$ is a graded vertex operator coalgebra with
    \begin{eqnarray*}
     &&c_{\mathcal {V}'}=\textbf{1}^*,\rho_{\mathcal {V}'}=\omega^*,\\
     &&(\Yup_{\mathcal {V}'}(z)f,u\otimes v)=(f,Y(u,z)v),
    \end{eqnarray*}
where $f\in \mathcal {V}',u,v\in\mathcal {V}$, $\textbf{1}^*,\omega^*$ are the dual elements of $\textbf{1},\omega.$

(ii) $(V'=\oplus_{s\in\mathbb{Z}}V_s',Y_{V'}(\cdot,z),\textbf{1}_{V'},\omega_{V'})$ is a vertex operator algebra with
    \begin{eqnarray*}
     &&\textbf{1}_{V'}=c|_{V_0},\omega_{V'}=\rho|_{V_2},\\
     &&(Y(f,z)g,v)=(f\otimes g,\Yup(z)v),
    \end{eqnarray*}
where $f,g\in V',v\in V.$

(iii) $\mathcal {N}'=\oplus_{s\in\mathbb{N}}N(s)^*$ is an admissible $\mathcal {V}'$-comodule with $$(\Yup_{\mathcal {N}'}(z)n^*,v\otimes n)=(n^*,Y_\mathcal {N}(v,z)n),$$where $n^*\in \mathcal {N}',v\in\mathcal {V},n\in \mathcal {N}.$

(iv) $\mathcal {M}'=\oplus_{s\in\mathbb{N}}M(s)^*$ is an admissible $V'$-module with $$(Y_{\mathcal {M}'}(f,z)m^*,m)=(f\otimes m^*,\Yup_\mathcal {M}(z)m),$$where $f\in V',m^*\in \mathcal {M}',m\in \mathcal {M}.$
\end{Proposition}

\begin{Corollary}\label{Coro2.13}
 Let $\mathcal {M}$ be an admissible $V$-comodule. Then $\mathcal {M}$ and $\mathcal {M}'$ are modules of Virasoro algebra $\oplus \mathbb{C}L_V(k)\oplus\mathbb{C}d$, Furthermore, for $m\in\mathcal {M},m^*\in\mathcal {M}'$, we have $$(L_V(k)m,m^*)=(m,L_{V}(-k)m^*),$$for all $k\in\mathbb{Z}$.
\end{Corollary}

\section{Contragredient comodules}
In this section, we give the construction of $V$-comodule structure on the graded dual $\mathcal {M}'$ for a graded vertex operator coalgebra and an admissible $V$-comodule $\mathcal {M}$. From now on, we always assume $\rho\in V'$ is homogeneous, i.e., $\rho\in V_2^*.$

First, for a linear map $\phi:V\rightarrow V$, let $z^\phi:V\rightarrow V\{z\}$ be a linear map which is defined as $$z^\phi(v)=z^{\lambda}(v),$$ where $v$ is an eigenvector of $\phi$ with eigenvalue $\lambda$, and $V\{z\}=\{\sum_{k\in\mathbb{C}}v_kz^k|v_k\in V\}.$

\begin{Lemma}\label{Lm3.1}
Let $(V, \Yup(z), \rho,c)$ be a graded vertex operator coalgebra. Then we have following identities:

(i) $L_V(0)\rho=2\rho,~L_V(2)\rho=\frac{1}{2}dc,~L_V(n)\rho=0,$ where $n>0,n\neq2.$

(ii) $L_V(n)c=0,$ where $n\geq0.$
\end{Lemma}
\proof By Proposition \ref{Prop2.12}, we know $V'$ is a vertex operator algebra, the vacuum vector is $c$ and Virasoro vector is $\rho.$ Now above identities follow from Jacobi identity of vertex operator algebra $V'.$                        $\hfill\Box$

\begin{Lemma}\label{Lm3.2}
Let $(V, \Yup(z), \rho,c)$ be a graded vertex operator coalgebra. Then, we have the following identity
 \begin{eqnarray*}
  e^{zL_V(-1)}L_V(1)=L_V(1)e^{zL_V(-1)}-2zL_V(0)e^{zL_V(-1)}+z^2e^{zL_V(-1)}.
 \end{eqnarray*}
\end{Lemma}
\proof Using relations between $L_V(-1),L_V(0),L_V(1)$, we have
 \begin{eqnarray*}
  &&e^{zL_V(-1)}L_V(1)=L_V(1)e^{zL_V(-1)}+\sum_i\frac{z^i}{i!}[L_V(-1)^i,L_V(1)]\\
  &=&L_V(1)e^{zL_V(-1)}+\sum_i\frac{z^i}{i!}(-2L_V(0)L_V(-1)^{i-1}-2L_V(-1)L_V(0)L_V(-1)^{i-2}\\
  &&~~~~~~~~~~-\cdots-2L_V(-1)^{i-1}L_V(0))\\
  &=&L_V(1)e^{zL_V(-1)}+\sum_i\frac{z^i}{i!}(-2iL_V(0)L_V(-1)^{i-1}+i(i-1)L_V(-1)^{i-1})\\
  &=&L_V(1)e^{zL_V(-1)}-2zL_V(0)e^{zL_V(-1)}+z^2L_V(-1)e^{zL_V(-1)}.
 \end{eqnarray*}
This completes the proof.                                     $\hfill\Box$

\begin{Lemma}\label{Lm3.3}
Let $(V, \Yup(z), \rho,c)$ be a graded vertex operator coalgebra. Then, we have the following identity
 \begin{eqnarray*}
 \Yup(z_1)\circ z^{L_V(0)}=z^{L_V(0)\otimes L_V(0)}\circ\Yup(zz_1),
 \end{eqnarray*}
 or, equivalently
  \begin{eqnarray*}
 \Id\otimes z^{-L_V(0)}\circ\Yup(z_1)\circ z^{L_V(0)}=z^{L_V(0)}\otimes\Id\circ\Yup(zz_1).
 \end{eqnarray*}
\end{Lemma}
\proof From \cite{W}, we know $\Delta_i\circ L_V(0)=(L_V(0)\otimes L_V(0)-i-1)\circ\Delta_i$. Thus
 \begin{eqnarray*}
 &&\Yup(z_1)\circ z^{L_V(0)}=\sum_i\Delta_i\circ z^{L_V(0)}z_1^{-i-1}\\
 &=&\sum_iz^{L_V(0)\otimes L_V(0)-i-1}\circ\Delta_iz_1^{-i-1}=z^{L_V(0)\otimes L_V(0)}\circ\Yup(zz_1).
 \end{eqnarray*}
Now, the second identity is trivial. This completes the proof.                                                                   $\hfill\Box$

\begin{Lemma}\label{Lm3.4}
Let $(V, \Yup(z), \rho,c)$ be a graded vertex operator coalgebra. Then, we have the following identity
 \begin{eqnarray*}
 \Yup(z_2)\circ e^{zL_V(-1)}=e^{z(L_V(-1)+2z_2L_V(0)+z_2^2L_V(1))}\otimes e^{ zL_V(-1)}\circ\Yup(z_2),
 \end{eqnarray*}
 or, equivalently
  \begin{eqnarray*}
 \Id\otimes e^{-zL_V(-1)}\circ\Yup(z_2)\circ e^{zL_V(-1)}=e^{z(L_V(-1)+2z_2L_V(0)+z_2^2L_V(1))}\otimes \Id\circ\Yup(z_2).
 \end{eqnarray*}
\end{Lemma}
\proof First, applying $\Res_{z_0,z_1}z_1^2\rho\otimes\Id\otimes\Id$ to Jacobi identity \ref{JI3.4}, we get
 \begin{eqnarray*}
 &&\Yup(z_2)\circ L_V(-1)-\Id\otimes L_V(-1)\circ\Yup(z_2)\\
 &=&(L_V(-1)+2z_2L_V(0)+z_2^2L_V(1))\otimes\Id \circ\Yup(z_2).
 \end{eqnarray*}
Hence, we have
  \begin{eqnarray*}
 &&\Yup(z_2)\circ e^{zL_V(-1)}\\
 &=&\Id\otimes e^{ zL_V(-1)}\circ e^{z(L_V(-1)+2z_2L_V(0)+z_2^2L_V(1))}\otimes\Id \circ\Yup(z_2)\\
 &=&e^{z(L_V(-1)+2z_2L_V(0)+z_2^2L_V(1))}\otimes e^{ zL_V(-1)}\circ\Yup(z_2).
 \end{eqnarray*}
Now, the second identity is trivial. This completes the proof.                                                                   $\hfill\Box$

\begin{Lemma}\label{Lm3.5}
Let $(V, \Yup(z), \rho,c)$ be a graded vertex operator coalgebra. Then, we have the following identity
 \begin{eqnarray*}
 \Id\otimes e^{-zL_V(-1)}\circ\Yup(z_2)\circ e^{zL_V(-1)}=(1-zz_2)^{-2L_V(0)}e^{z(1-zz_2)L_V(-1)}\otimes\Id\circ\Yup (\frac{z_2}{1-zz_2}).
 \end{eqnarray*}
\end{Lemma}
\proof From $L_V(1)$-derivation property, it is obvious that $$\Id\otimes e^{\frac{zz_2^2}{1-zz_2}L_V(1)}\circ\Yup (\frac{z_2}{1-zz_2})\circ e^{\frac{-zz_2^2}{1-zz_2}L_V(1)}=e^{\frac{-zz_2^2}{1-zz_2}L_V(1)}\otimes\Id\circ\Yup (\frac{z_2}{1-zz_2})=\Yup (z_2).$$
Using above Lemma, we just need to show
$$e^{z(L_V(-1)+2z_2L_V(0)+z_2^2L_V(1))}\circ e^{\frac{-zz_2^2}{1-zz_2}L_V(1)}=e^{z(1-zz_2)L_V(-1)}(1-zz_2)^{-2L_V(0)}$$
holds on $V$. This is equivalently to show
$$e^{\frac{-zz_2^2}{1-zz_2}L_{V'}(-1)}\circ e^{z(L_{V'}(1)+2z_2L_{V'}(0)+z_2^2L_{V'}(-1))}=e^{z(1-zz_2)L_{V'}(1)}(1-zz_2)^{-2L_{V'}(0)}$$
holds on $V'$. This identity has proven in \cite{FHL}. Thus we are done.                                         $\hfill\Box$

\begin{Theorem}\label{Thm3.6}
Let $V$ be a graded vertex operator coalgebra and $(\mathcal {M},\Yup_\mathcal {M}(z))$ an admissible $V$-comodule such that each homogeneous subspace is of finite dimensional. Let $\mathcal {M}'$ be the graded dual space of $\mathcal {M}$, define $\Yup'_{\mathcal {M}'}(z):\mathcal {M}'\rightarrow V\otimes\mathcal {M}'[[z,z^{-1}]]$ by$$(\Yup'_{\mathcal {M}'}(z)m^*,f\otimes m)=(f\otimes m^*,(-z^{-2})^{L_V(0)}e^{zL_V(-1)}\otimes \Id\circ\Yup_\mathcal {M}(z^{-1})m),$$where $m\in\mathcal {M},m^*\in\mathcal {M}',f\in V'.$ Then $(\mathcal {M}',\Yup'_{\mathcal {M}'}(z))$ is also an admissible $V$-comodule, called contragredient comodule of $\mathcal {M}$.
\end{Theorem}
\proof Let $\mathcal {M}=\oplus_{s\geq0}M(s)$. Then $\mathcal {M}'=\oplus_{s\geq0} M(s)^*$.

(i) First, we show $\mathcal {M}'$ satisfies the Virasoro properties. Let $m^*\in\mathcal {M}'$ and $m\in \mathcal {M}$, we have
    \begin{eqnarray*}
     &&((\rho\otimes\Id)\circ\Yup'_{\mathcal {M}'}(z)m^*,m)=(\Yup'_{\mathcal {M}'}(z)m^*,\rho\otimes m)\\
     &=&(\rho\otimes m^*,(-z^{-2})^{L_V(0)}e^{zL_V(-1)}\otimes \Id\circ\Yup_\mathcal {M}(z^{-1})m)\\
     &=&(e^{zL_V(1)}(-z^{-2})^{L_V(0)}\rho\otimes m^*,\Yup_\mathcal {M}(z^{-1})m)\\
     &=&(\frac{1}{z^4}\rho\otimes m^*,\Yup_\mathcal {M}(z^{-1})m)=(\frac{1}{z^4} m^*,\rho\otimes \Id\circ\Yup_\mathcal {M}(z^{-1})m)\\
     &=&\sum_k(\frac{1}{z^4}m^*,L_V(k)z^{-k+2}m)=(\sum_kL_V(-k)z^{-k-2}m^*,m)\\
     &=&(\sum_kL_V(k)z^{k-2}m^*,m).
    \end{eqnarray*}
Hence, we get $(\rho\otimes\Id)\circ\Yup'_{\mathcal {M}'}(z)=\sum_kL_V(k)z^{k-2}.$ From Corollary \ref{Coro2.13}, we know the Virasoro properties hold.

(ii) Second, for $m^*\in\mathcal {M}'$ and $m\in \mathcal {M}$, we have
    \begin{eqnarray*}
     &&((c\otimes\Id)\circ\Yup'_{\mathcal {M}'}(z)m^*,m)=(\Yup'_{\mathcal {M}'}(z)m^*,c\otimes m)\\
     &=&(c\otimes m^*,(-z^{-2})^{L_V(0)}e^{zL_V(-1)}\otimes \Id\circ\Yup_\mathcal {M}(z^{-1})m)\\
     &=&(e^{zL_V(1)}(-z^{-2})^{L_V(0)}c\otimes m^*,\Yup_\mathcal {M}(z^{-1})m)\\
     &=&(c\otimes m^*,\Yup_\mathcal {M}(z^{-1})m)=(m^*,c\otimes \Id\circ\Yup_\mathcal {M}(z^{-1})m)\\
     &=&(m^*,m).
    \end{eqnarray*}
Hence, we get $(c\otimes\Id)\circ\Yup_{\mathcal {M}'}(z)=\Id.$

(iii) Third, let $m^*\in\mathcal {M}'$ be homogeneous of weight $s$. For any $k\in\mathbb{Z}$, suppose $m\in M(t),f\in V'$, we have
    \begin{eqnarray*}
    &&(\Res_zz^k\Yup'_{\mathcal {M}'}(z)m^*,f\otimes m)=\Res_zz^k(f\otimes m^*,(-z^{-2})^{L_V(0)}e^{zL_V(-1)}\otimes \Id\circ\Yup_\mathcal {M}(z^{-1})m)\\
    &=&\sum_{i\geq0,j\geq0,l}\Res_zz^{k+i}(f\otimes m^*,(-z^{-2})^{L_V(0)}\frac{(L_V(-1))^i}{i!}m_{l+1+t-j}'\otimes m_j'')z^{l+1})\\
    &=&\sum_{i\geq0,j\geq0,l}\Res_zz^{k+i-2(l+1+t-j+i)}(f\otimes m^*,\frac{(L_V(-1))^i}{i!}(-1)^{l+1+t-j}m_{l+1+t-j}'\otimes m_j'')z^{l+1})\\
    &=&\sum_{i\geq0}(f\otimes m^*,\frac{(L_V(-1))^i}{i!}(-1)^{k-i+s+1-t}m_{k-i+s+1-t}'\otimes m_s'')).
    \end{eqnarray*}
Hence, we get $f\in V_{k+s+1-t}'$. Therefore, we have $$\Delta_{\mathcal {M}',k}(m^*)\in\oplus_{t\in\mathbb{N}} V_{k+s+1-t}\otimes M(t)^*=(V\otimes\mathcal {M}')_{k+s+1},$$ i.e., $\mathcal {M}'$ satisfies the gradation condition.

Since $V$ and $\mathcal {M}$ are lower truncated and each homogeneous subspace is of finite dimensional, it is obvious that $\Delta_{\mathcal {M}',k}(m^*)$ is a finite sum and $\Delta_{\mathcal {M}',k}(m^*)=0$ for $k<<0.$

(iv) Now we prove the $L_V(1)$-derivation property. Let $f\in V',m^*\in\mathcal {M}',m\in\mathcal {M}$, we have
\begin{eqnarray*}
&&\frac{d}{dz}(\Yup'_{\mathcal {M}'}(z)m^*,f\otimes m)=\frac{d}{dz}(f\otimes m^*,(-z^{-2})^{L_V(0)}e^{zL_V(-1)}\otimes \Id\circ\Yup_\mathcal {M}(z^{-1})m)\\
&=&(f\otimes m^*,-2L_V(0)z^{-1}(-z^{-2})^{L_V(0)}e^{zL_V(-1)}\otimes \Id\circ\Yup_\mathcal {M}(z^{-1})m)\\
&&+(f\otimes m^*,(-z^{-2})^{L_V(0)}e^{zL_V(-1)}L_V(-1)\otimes \Id\circ\Yup_\mathcal {M}(z^{-1})m)\\
&&-z^{-2}(f\otimes m^*,(-z^{-2})^{L_V(0)}e^{zL_V(-1)}\otimes \Id\circ\frac{d}{dz^{-1}}\Yup_\mathcal {M}(z^{-1})m).
\end{eqnarray*}
Using $L_V(1)$-derivation property for $\mathcal {M}$, we have
\begin{eqnarray*}
&&-z^{-2}(f\otimes m^*,(-z^{-2})^{L_V(0)}e^{zL_V(-1)}\otimes \Id\circ\frac{d}{dz^{-1}}\Yup_\mathcal {M}(z^{-1})m)\\
&=&-z^{-2}(f\otimes m^*,(-z^{-2})^{L_V(0)}e^{zL_V(-1)}L_V(1)\otimes \Id\circ\Yup_\mathcal {M}(z^{-1})m)\\
&=&-z^{-2}(f\otimes m^*,(-z^{-2})^{L_V(0)}(L_V(1)e^{zL_V(-1)}-2zL_V(0)e^{zL_V(-1)}+z^2L_V(-1)e^{zL_V(-1)})\\
&&~~~~~~~~~~~~\otimes \Id\circ\Yup_\mathcal {M}(z^{-1})m)\\
&=&(f\otimes m^*,(L_V(1)(-z^{-2})^{L_V(0)}e^{zL_V(-1)}+2z^{-1}(-z^{-2})^{L_V(0)}L_V(0)e^{zL_V(-1)}\\
&&~~~~~~~~~~~~-(-z^{-2})^{L_V(0)}L_V(-1)e^{zL_V(-1)})\otimes \Id\circ\Yup_\mathcal {M}(z^{-1})m).
\end{eqnarray*}
From above two identities, we have
\begin{eqnarray*}
&&\frac{d}{dz}(\Yup'_{\mathcal {M}'}(z)m^*,f\otimes m)\\
&=&(f\otimes m^*,L_V(1)(-z^{-2})^{L_V(0)}e^{zL_V(-1)}\otimes \Id\circ\Yup_\mathcal {M}(z^{-1})m)\\
&=&(L_V(-1)f\otimes m^*,(-z^{-2})^{L_V(0)}e^{zL_V(-1)}\otimes \Id\circ\Yup_\mathcal {M}(z^{-1})m)\\
&=&(\Yup'_{\mathcal {M}'}(z)m^*,L_V(-1)f\otimes m)\\
&=&((L_V(1)\otimes \Id)\circ\Yup'_{\mathcal {M}'}(z)m^*,f\otimes m).
\end{eqnarray*}
This proves the $L_V(1)$-derivation property.

(v) Last, we prove the Jacobi identity. Let $f,g\in V',m^*\in\mathcal {M}',m\in\mathcal {M}$, we need to show
    \begin{eqnarray}\label{ID3.11}
      &&(z_0^{-1}\delta(\frac{z_1-z_2}{z_0})(\Id\otimes \Yup'_{\mathcal {M}'} (z_2))\circ \Yup'_{\mathcal {M}'} (z_1),f\otimes g\otimes m)\nonumber\\
      &&~~~~~~~~~~-(z_0^{-1}\delta(\frac{z_2-z_1}{-z_0})(\tau\otimes\Id )\circ(\Id\otimes \Yup'_{\mathcal {M}'} (z_1))\circ \Yup'_{\mathcal {M}'} (z_2),f\otimes g\otimes m)\nonumber\\
      &&=(z_1^{-1}\delta(\frac{z_2+z_0}{z_1})(\Yup (z_0)\otimes \Id )\circ \Yup'_{\mathcal {M}'} (z_2),f\otimes g\otimes m).
    \end{eqnarray}

First by definition of $\Yup_{\mathcal {M}'}$, we have
\begin{eqnarray*}
&&(\Id\otimes \Yup'_{\mathcal {M}'}(z_2)\circ\Yup'_{\mathcal {M}'}(z_1)m^*,f\otimes g\otimes m)\\
&=&(g\otimes f\otimes m^*,\Id\otimes(-z_1^{-2})^{L_V(0)}e^{z_1L_V(-1)}\otimes\Id\circ\Id\otimes\Yup_\mathcal {M}(z_1^{-1})\\
&&~~~~~~~\circ (-z_2^{-2})^{L_V(0)}e^{z_2L_V(-1)}\otimes\Id\circ\Yup_\mathcal {M}(z_2^{-1})m)\\
&=&(g\otimes f\otimes m^*,\Id\otimes(-z_1^{-2})^{L_V(0)}e^{z_1L_V(-1)}\otimes\Id\circ(-z_2^{-2})^{L_V(0)}e^{z_2L_V(-1)}\otimes\Id\otimes\Id\\
&&~~~~~~~\circ \Id\otimes\Yup_\mathcal {M}(z_1^{-1})\circ\Yup_\mathcal {M}(z_2^{-1})m)\\
&=&(g\otimes f\otimes m^*,(-z_2^{-2})^{L_V(0)}e^{z_2L_V(-1)}\otimes(-z_1^{-2})^{L_V(0)}e^{z_1L_V(-1)}\otimes\Id\\
&&~~~~~~~\circ\Id\otimes\Yup_\mathcal {M}(z_1^{-1})\circ\Yup_\mathcal {M}(z_2^{-1})m)\\
&=&(f\otimes g\otimes m^*,(-z_1^{-2})^{L_V(0)}e^{z_1L_V(-1)}\otimes(-z_2^{-2})^{L_V(0)}e^{z_2L_V(-1)}\otimes\Id\\
&&~~~~~~~\circ\tau\otimes\Id\circ\Id\otimes\Yup_\mathcal {M}(z_1^{-1})\circ\Yup_\mathcal {M}(z_2^{-1})m),
\end{eqnarray*}
and
\begin{eqnarray*}
&&(\tau\otimes\Id\circ\Id\otimes \Yup'_{\mathcal {M}'}(z_1)\circ\Yup'_{\mathcal {M}'}(z_2)m^*,f\otimes g\otimes m)\\
&=&(f\otimes g\otimes m^*,(-z_1^{-2})^{L_V(0)}e^{z_1L_V(-1)}\otimes(-z_2^{-2})^{L_V(0)}e^{z_2L_V(-1)}\otimes\Id\\
&&~~~~~~~\circ \Id\otimes\Yup_\mathcal {M}(z_2^{-1})\circ\Yup_\mathcal {M}(z_1^{-1})m).
\end{eqnarray*}
Since $\mathcal {M}$ is an admissible $V$-comodule, we have the following identity
    \begin{eqnarray*}
      &&(\frac{-z_0}{z_1z_2})^{-1}\delta(\frac{z_1^{-1}-z_2^{-1}}{-z_0/z_1z_2})(\Id\otimes \Yup_\mathcal {M} (z_2^{-1}))\circ \Yup_\mathcal {M} (z_1^{-1})m\\
      &&~~~~~~~~~~-(\frac{-z_0}{z_1z_2})^{-1}\delta(\frac{z_2^{-1}-z_1^{-1}}{z_0/z_1z_2})(\tau\otimes\Id )\circ(\Id\otimes \Yup_\mathcal {M} (z_1^{-1}))\circ \Yup_\mathcal {M} (z_2^{-1})m\\
      &&=z_2\delta(\frac{z_1^{-1}+z_0/z_1z_2}{z_2^{-1}})(\Yup (-z_0/z_1z_2)\otimes \Id )\circ \Yup_\mathcal {M} (z_2^{-1})m.
    \end{eqnarray*}
This is equivalent to
    \begin{eqnarray*}
      &&-z_0^{-1}\delta(\frac{z_2-z_1}{-z_0})(\Id\otimes \Yup_\mathcal {M} (z_2^{-1}))\circ \Yup_\mathcal {M} (z_1^{-1})m\\
      &&~~~~~~~~~~+z_0^{-1}\delta(\frac{z_1-z_2}{z_0})(\tau\otimes\Id )\circ(\Id\otimes \Yup_\mathcal {M} (z_1^{-1}))\circ \Yup_\mathcal {M} (z_2^{-1})m\\
      &&=z_1^{-1}\delta(\frac{z_2+z_0}{z_1})(\Yup (-z_0/z_1z_2)\otimes \Id )\circ \Yup_\mathcal {M} (z_2^{-1})m.
    \end{eqnarray*}
Now, the left hand side of identity (\ref{ID3.11}) equals to
    \begin{eqnarray*}
      &&(f\otimes g\otimes m^*,(-z_1^{-2})^{L_V(0)}e^{z_1L_V(-1)}\otimes(-z_2^{-2})^{L_V(0)}e^{z_2L_V(-1)}\otimes\Id\circ  \\
      &&~~~~~~~~~~~~~~ z_1^{-1}\delta(\frac{z_2+z_0}{z_1})(\Yup (-z_0/z_1z_2)\otimes \Id )\circ \Yup_\mathcal {M} (z_2^{-1})m).
    \end{eqnarray*}
Now, by definition, we have
\begin{eqnarray*}
&&(\Yup(z_0)\otimes\Id \circ\Yup'_{\mathcal {M}'}(z_2)m^*,f\otimes g\otimes m)\\
&=&(f\otimes g\otimes m^*,\Yup(z_0)\otimes\Id\circ (-z_2^{-2})^{L_V(0)}e^{z_2L_V(-1)}\otimes\Id\circ\Yup_\mathcal {M}(z_2^{-1})m).
\end{eqnarray*}
Hence, it is enough to show
    \begin{eqnarray*}
    &&\Yup(z_0)\otimes\Id\circ (-z_2^{-2})^{L_V(0)}e^{z_2L_V(-1)}\otimes\Id\circ\Yup_\mathcal {M}(z_2^{-1})m\\
      &=&(-(z_2+z_0)^{-2})^{L_V(0)}e^{(z_2+z_0)L_V(-1)}\otimes(-z_2^{-2})^{L_V(0)}e^{z_2L_V(-1)}\otimes\Id\circ  \\
      &&~~~~~~~~~~~~~~ \Yup (-z_0/(z_2+z_0)z_2)\otimes \Id \circ \Yup_\mathcal {M} (z_2^{-1})m.
    \end{eqnarray*}
This is equivalent to show
    \begin{eqnarray*}
    &&\Yup(z_0)\circ (-z_2^{-2})^{L_V(0)}e^{z_2L_V(-1)}\\
      &=&(-(z_2+z_0)^{-2})^{L_V(0)}e^{(z_2+z_0)L_V(-1)}\otimes(-z_2^{-2})^{L_V(0)}e^{z_2L_V(-1)}\circ \Yup (-z_0/(z_2+z_0)z_2),
    \end{eqnarray*}
or, equivalently,
    \begin{eqnarray*}
    &&\Id\otimes e^{-z_2L_V(-1)}(-z_2^{2})^{L_V(0)}\circ \Yup(z_0)\circ (-z_2^{-2})^{L_V(0)}e^{z_2L_V(-1)}\\
      &=&(-(z_2+z_0)^{-2})^{L_V(0)}e^{(z_2+z_0)L_V(-1)}\otimes\Id\circ \Yup (-z_0/(z_2+z_0)z_2).
    \end{eqnarray*}
Now, from Lemma \ref{Lm3.3} and Lemma \ref{Lm3.5}, we have
    \begin{eqnarray*}
    &&\Id\otimes e^{-z_2L_V(-1)}(-z_2^{2})^{L_V(0)}\circ \Yup(z_0)\circ (-z_2^{-2})^{L_V(0)}e^{z_2L_V(-1)}\\
    &=&(-z_2^{-2})^{L_V(0)}\otimes \Id\circ \Id\otimes e^{-z_2L_V(-1)}\circ\Yup (-z_0z_2^{-2})\circ e^{z_2L_V(-1)}\\
    &=&(-z_2^{-2})^{L_V(0)}\otimes \Id\circ (1+z_0/z_2)^{-2L_V(0)}e^{z_2(1+z_0/z_2)L_V(-1)}\otimes\Id\circ \Yup (\frac{-z_0}{z_2(z_2+z_0)})\\
      &=&(-(z_2+z_0)^{-2})^{L_V(0)}e^{(z_2+z_0)L_V(-1)}\otimes\Id\circ \Yup (-z_0/(z_2+z_0)z_2).
    \end{eqnarray*}
This completes the proof.                                                  $\hfill\Box$

\begin{Proposition}\label{Prop3.7}
Let $V$ be a graded vertex operator coalgebra and $(\mathcal {M},\Yup_\mathcal {M}(z))$ an admissible $V$-comodule such that each homogeneous subspace is of finite dimensional. Then, $\mathcal {M}''$ is isomorphic to $\mathcal {M}$ as admissible $V$-comodules.
\end{Proposition}
\proof It is enough to prove $\Yup_{\mathcal {M}} (z)=\Yup_{\mathcal {M}''}''(z)$.

Let $m\in \mathcal {M}=\mathcal {M}'',m^*\in\mathcal {M}',f\in V',$ we have
    \begin{eqnarray*}
    &&(\Yup_{\mathcal {M}''}''(z)m,f\otimes m^*)\\
    &=&(f\otimes m,(-z^{-2})^{L_V(0)}e^{zL_V(-1)}\otimes \Id\circ\Yup'_{\mathcal {M}'}(z^{-1})m^*)\\
    &=&(e^{zL_{V'}(1)}(-z^{-2})^{L_{V'}(0)}f\otimes m^*,(-z^{2})^{L_V(0)}e^{z^{-1}L_V(-1)}\otimes \Id\circ\Yup_\mathcal {M}(z)m)\\
    &=&(e^{z^{-1}L_{V'}(1)}(-z^{2})^{L_{V'}(0)}e^{zL_{V'}(1)}(-z^{-2})^{L_{V'}(0)}f\otimes m^*,\Yup_\mathcal {M}(z)m)\\
    &=&(\Yup_{\mathcal {M}}(z)m,f\otimes m^*).
    \end{eqnarray*}
This completes the proof.                                                                    $\hfill\Box$

\section{Cointertwining operators among comodules}
In this section, we study the cointertwining operators among admissible $V$-comodules.

\begin{Definition}
Let $\mathcal {M}^1,\mathcal {M}^2,\mathcal {M}^3$ be three admissible $V$-comodules. A cointertwining operator $\mathcal {W}(z)$ of type $\tbinom{\mathcal {M}^1~~\mathcal {M}^2}{\mathcal {M}^3}$ is a linear map
   \begin{eqnarray*}
  \mathcal {W}(z):~\mathcal {M}^3&\rightarrow&~\mathcal {M}^1\otimes \mathcal {M}^2\{z\},\\
  m^3&\mapsto& \mathcal {W}(z)m^3=\sum_{t\in\mathbb{C}}\Delta_t(m^3)z^{-t-1},
  \end{eqnarray*}
  where $\Delta_t(m^3)\in \mathcal {M}^1\otimes \mathcal {M}^2$, satisfying the following conditions:

  (i) For any $m^3\in \mathcal {M}^3$, $\Delta_t(m^3)=0$ for $\R~ t\in\mathbb{R}$ sufficiently small.

  (ii) The $L_V(1)$-derivation $\frac{d}{dz}\mathcal {W}(z)=L_V(1)\otimes \Id\circ \mathcal {W}(z)$ holds.

  (iii) The following Jacobi identity
  \begin{eqnarray}\label{JI4.1}
      &&z_0^{-1}\delta(\frac{z_1-z_2}{z_0})(\Id\otimes \mathcal {W} (z_2))\circ \Yup_{\mathcal {M}^3} (z_1)\nonumber\\
      &&~~~~~~~~~~-z_0^{-1}\delta(\frac{z_2-z_1}{-z_0})(\tau\otimes\Id )\circ(\Id\otimes \Yup_{\mathcal {M}^2} (z_1))\circ \mathcal {W} (z_2)\nonumber\\
      &&=z_1^{-1}\delta(\frac{z_2+z_0}{z_1})(\Yup_{\mathcal {M}^1} (z_0)\otimes \Id )\circ \mathcal {W} (z_2)
  \end{eqnarray}
holds.

  Let $I_V\tbinom{\mathcal {M}^1~~\mathcal {M}^2}{\mathcal {M}^3}$ be the set of all cointertwining operator of type $\tbinom{\mathcal {M}^1~~\mathcal {M}^2}{\mathcal {M}^3}$. Then it is a vector space. Set $N_{\mathcal {M}^3}^{\mathcal {M}^1,\mathcal {M}^2}=\dim I_V\tbinom{\mathcal {M}^1~~\mathcal {M}^2}{\mathcal {M}^3}$, called the fusion rule of type $\tbinom{\mathcal {M}^1~~\mathcal {M}^2}{\mathcal {M}^3}$.
\end{Definition}

\begin{Remark}
(i) Let $(V,\Yup (z),c,\rho)$ be a graded vertex operator coalgebra, then $\Yup (z)$ is a cointertwining operator of type $\tbinom{V~~V}{V}$.

(ii) Let $\mathcal {M},\Yup_\mathcal {M}(z)$ be an admissible $V$-comodule, then $\Yup_\mathcal {M}(z)$ is a cointertwining operator of type $\tbinom{V~~\mathcal {M}}{\mathcal {M}}$.
\end{Remark}

\begin{Proposition}\label{Prop4.3}
Let $\mathcal {W}(z)$ be a cointertwining operator of type $\tbinom{\mathcal {M}^1~~\mathcal {M}^2}{\mathcal {M}^3}$. Define $\Omega(\mathcal {W})$ by $$\Omega(\mathcal {W})(z)=\tau\circ\mathcal {W}(-z)\circ e^{zL_V(1)}.$$ Then, $\Omega(\mathcal {W})$ is a cointertwining operator of type $\tbinom{\mathcal {M}^2~~\mathcal {M}^1}{\mathcal {M}^3}$. Furthermore, we have $$N_{\mathcal {M}^3}^{\mathcal {M}^1,\mathcal {M}^2}=N_{\mathcal {M}^3}^{\mathcal {M}^2,\mathcal {M}^1}.$$
\end{Proposition}
\proof Applying $\Res_{z_0,z_1}\rho\otimes\Id\otimes\Id$ to Jacobi identity \ref{JI4.1}, we get
$$\mathcal {W}(z_2)\circ L_V(1)- \Id\otimes L_V(1)\circ\mathcal {W}(z_2)=L_V(1)\otimes\Id\circ\mathcal {W}(z_2).$$
Now, we have
 \begin{eqnarray*}
 &&\frac{d}{dz}\Omega(\mathcal {W})(z)\\
 &=&\tau\circ\frac{d}{dz}\mathcal {W}(-z)\circ e^{zL_V(1)}+\tau\circ\mathcal {W}(-z)\circ L_V(1)e^{zL_V(1)}\\
 &=&-\tau\circ L_V(1)\otimes\Id\circ\mathcal {W}(-z)\circ e^{zL_V(1)}+\tau\circ\mathcal {W}(-z)\circ L_V(1)e^{zL_V(1)}\\
 &=&\tau\circ\Id\otimes L_V(1)\circ\mathcal {W}(-z)\circ e^{zL_V(1)}\\
 &=&L_V(1)\otimes\Id\circ\tau\circ\mathcal {W}(-z)\circ e^{zL_V(1)}\\
 &=&L_V(1)\otimes\Id\circ\Omega(\mathcal {W})(z).
  \end{eqnarray*}
Thus we prove the derivation property.

Now, we show the Jacobi identity. It is enough to prove
  \begin{eqnarray*}
      &&z_0^{-1}\delta(\frac{z_1-z_2}{z_0})\Id\otimes \Omega(\mathcal {W}) (z_2)\circ \Yup_{\mathcal {M}^3} (z_1)\\
      &&~~~~~~~~~~-z_0^{-1}\delta(\frac{z_2-z_1}{-z_0})\tau\otimes\Id \circ\Id\otimes \Yup_{\mathcal {M}^1} (z_1)\circ \Omega(\mathcal {W}) (z_2)\\
      &&=z_1^{-1}\delta(\frac{z_2+z_0}{z_1})\Yup_{\mathcal {M}^2} (z_0)\otimes \Id \circ \Omega(\mathcal {W}) (z_2).
  \end{eqnarray*}
First, by definitions of $\Omega$ and $\tau$, and Proposition \ref{Prop2.15}, we have
\begin{eqnarray*}
\Id\otimes \Omega(\mathcal {W}) (z_2)\circ \Yup_{\mathcal {M}^3} (z_1)&=&\Id\otimes\tau\circ\Id\otimes\mathcal {W}(-z_2)\circ \Id\otimes e^{z_2L_V(1)}\circ \Yup_{\mathcal {M}^3} (z_1)\\
&=&\Id\otimes\tau\circ\Id\otimes\mathcal {W}(-z_2)\circ  \Yup_{\mathcal {M}^3} (z_1-z_2)\circ e^{z_2L_V(1)},\\
\tau\otimes\Id \circ\Id\otimes \Yup_{\mathcal {M}^1} (z_1)\circ \Omega(\mathcal {W}) (z_2)&=&\tau\otimes\Id \circ\Id\otimes \Yup_{\mathcal {M}^1} (z_1)\circ\tau\circ\mathcal {W}(-z_2)\circ e^{z_2L_V(1)}\\
&=&\Id\otimes\tau\circ\Yup_{\mathcal {M}^1} (z_1)\otimes\Id\circ\mathcal {W}(-z_2)\circ e^{z_2L_V(1)},\\
\Yup_{\mathcal {M}^2} (z_0)\otimes \Id \circ \Omega(\mathcal {W}) (z_2)&=&\Yup_{\mathcal {M}^2} (z_0)\otimes \Id \circ \tau\circ\mathcal {W}(-z_2)\circ e^{z_2L_V(1)}\\
&=&\Id\otimes\tau\circ\tau\otimes\Id\circ\Id\otimes\Yup_{\mathcal {M}^2} (z_0)\circ\mathcal {W}(-z_2)\circ e^{z_2L_V(1)}.
\end{eqnarray*}
Hence, we just need to prove
  \begin{eqnarray*}
      &&z_0^{-1}\delta(\frac{z_1-z_2}{z_0})\Id\otimes\mathcal {W}(-z_2)\circ  \Yup_{\mathcal {M}^3} (z_1-z_2)\\
      &&~~~~~~~~~~-z_0^{-1}\delta(\frac{z_2-z_1}{-z_0})\Yup_{\mathcal {M}^1} (z_1)\otimes\Id\circ\mathcal {W}(-z_2)\\
      &&=z_1^{-1}\delta(\frac{z_2+z_0}{z_1})\tau\otimes\Id\circ\Id\otimes\Yup_{\mathcal {M}^2} (z_0)\circ\mathcal {W}(-z_2).
  \end{eqnarray*}
From properties of $\delta$-function, this is equivalent to show
  \begin{eqnarray*}
      &&z_1^{-1}\delta(\frac{z_0+z_2}{z_1})\Id\otimes\mathcal {W}(-z_2)\circ  \Yup_{\mathcal {M}^3} (z_0)\\
      &&~~~~~~~~~~-z_1^{-1}\delta(\frac{-z_2-z_0}{-z_1})\tau\otimes\Id\circ\Id\otimes\Yup_{\mathcal {M}^2} (z_0)\circ\mathcal {W}(-z_2)\\
      &&=z_0^{-1}\delta(\frac{-z_2+z_1}{z_0})\Yup_{\mathcal {M}^1} (z_1)\otimes\Id\circ\mathcal {W}(-z_2).
  \end{eqnarray*}
But this identity is obvious since $\mathcal {W}(z)$ is a cointertwining operator of type $\tbinom{\mathcal {M}^1~~\mathcal {M}^2}{\mathcal {M}^3}$. Thus $\Omega(\mathcal {W})$ is a cointertwining operator of type $\tbinom{\mathcal {M}^2~~\mathcal {M}^1}{\mathcal {M}^3}$.

The second statement is trivial since $\Omega^2=\Id.$                                $\hfill\Box$

\begin{Proposition}\label{Prop4.4}
Let $\mathcal {W}(z)$ be a cointertwining operator of type $\tbinom{\mathcal {M}^1~~\mathcal {M}^2}{\mathcal {M}^3}$. Define $\mathcal {A}(\mathcal {W})$ by $$(\mathcal {A}(\mathcal {W})(z)m^{2*},m^{1*}\otimes m^{3})=(m^{1*}\otimes m^{2*},(-z^{-2})^{L_V(0)}e^{zL_V(-1)}\otimes \Id\circ\mathcal {W}(z^{-1})m^3).$$ Then, $\mathcal {A}(\mathcal {W})$ is a cointertwining operator of type $\tbinom{\mathcal {M}^1~~{\mathcal {M}^3}'}{{\mathcal {M}^2}'}$. Furthermore, we have$$N_{\mathcal {M}^3}^{\mathcal {M}^1,\mathcal {M}^2}=N_{{\mathcal {M}^2}'}^{\mathcal {M}^1,{\mathcal {M}^3}'}.$$
\end{Proposition}
\proof This is similar to the proof of Theorem \ref{Thm3.6} and Proposition \ref{Prop3.7}. So we omit it.                                                  $\hfill\Box$

\section{Cotensor decomposition}
In this section, we study the cotensor product of admissible comodules for a graded
vertex operator coalgebra. Throughout this section, let $V$ be a graded vertex operator
coalgebra and let $\mathcal{C}$ be the category of admissible $V$-comodules. We assume
that $V$ is simple and corational, then from \cite{W}, we know  $\mathcal{C}$ is semisimple with only finitely many isomorphism classes of simple
objects. We also assume that every object of $\mathcal{C}$ has finite-dimensional homogeneous subspaces and
that all fusion rules defined by cointertwining operators are
finite.

Let $K[\mathcal{C}]$ be the Grothendieck group of $\mathcal{C}$ and set
\[
\mathbb{K}[\mathcal{C}]=\mathbb{C}\otimes_{\mathbb{Z}}K[\mathcal{C}].
\]
For an admissible $V$-comodule $\mathcal {M}$, we denote its class in $\mathbb{K}[\mathcal{C}]$ by
$[\mathcal {M}]$.

\begin{Definition}\label{Prop5.2}
Let $\mathcal {M}$ be an admissible $V$-comodule. A cotensor decomposition of $\mathcal {M}$ is a formal object
\[
\Delta_{\boxtimesc}(\mathcal {M})
=
\bigoplus_{\alpha\in A} \mathcal {M}_\alpha^{(1)}\boxtimes \mathcal {M}_\alpha^{(2)},
\]
where $A$ is an index set,
together with cointertwining operators
\[
\mathcal {W}_\alpha(z)\in \mathcal{I}_V
\binom{\mathcal {M}_\alpha^{(1)}\quad \mathcal {M}_\alpha^{(2)}}{\mathcal {M}},
\]
satisfying the following universal property: for any admissible $V$-comodules $\mathcal {N}^1,\mathcal {N}^2$,
and cointertwining operator
$
\mathcal {W}^1(z)\in \mathcal{I}_V
\binom{\mathcal {N}^1\quad \mathcal {N}^2}{\mathcal {M}},
$
there exists a unique $\alpha$ and $V$-comodule homomorphisms
\[
\varphi_1:\mathcal {M}_\alpha^{(1)}\rightarrow \mathcal {N}^1,\qquad
\varphi_2:\mathcal {M}_\alpha^{(2)}\rightarrow \mathcal {N}^2,
\]
such that
\[
\mathcal {W}^1(z)=(\varphi_1\otimes \varphi_2)\circ \mathcal {W}_\alpha(z).
\]
\end{Definition}

\begin{Proposition}
Assume
$
\{\mathcal {M}^i\}_{i\in I}.
$
is a complete set of representatives of irreducible admissible $V$-comodules.
Let $\mathcal {M}$ be an admissible $V$-comodule. Suppose that the cotensor decomposition of $\mathcal {M}$
exists. Then,
we have
\[
\Delta_{\boxtimesc}(\mathcal {M})
=
\sum_{i,j\in I}
N^{\mathcal {M}^i,\mathcal {M}^j}_{\mathcal {M}}
\mathcal {M}^i\boxtimes \mathcal {M}^j.
\]
\end{Proposition}
\proof We need to prove that
\[
\bigoplus_{i,j\in I}
N^{\mathcal {M}^i,\mathcal {M}^j}_{\mathcal {M}}
\mathcal {M}^i\boxtimes \mathcal {M}^j
\]
satisfies the universal property of the cotensor decomposition of $\mathcal {M}$.

For each $i,j\in I$, let
$
\mathcal{W}^{i,j}_s(z),
\qquad
s=1,\ldots,N^{\mathcal {M}^i,\mathcal {M}^j}_{\mathcal {M}},
$
be a basis of the vector space
$
\mathcal{I}_V
\binom{\mathcal {M}^i,\mathcal {M}^j}{\mathcal {M}}.
$
Define the universal cointertwining operator by
\[
\mathcal{W}_\mathcal {M}(z)
=
\sum_{i,j\in I}
\sum_{s=1}^{N^{\mathcal {M}^i,\mathcal {M}^j}_{\mathcal {M}}}
\mathcal{W}^{i,j}_s(z),
\]
where $\mathcal{W}^{a,b}_s(z)$ is regarded as a cointertwining operator from $\mathcal {M}$ to
the $s$-th copy of $\mathcal {M}^i\otimes \mathcal {M}^j\{z\}$ inside
$
N_\mathcal {M}^{\mathcal {M}^i,\mathcal {M}^j}M^a\boxtimes M^b.
$

Let $\mathcal {N}^1,\mathcal {N}^2$ be admissible $V$-comodules and let
$
\mathcal{W}(z)
\in
\mathcal{I}_V
\binom{\mathcal {N}^1\quad \mathcal {N}^2}{\mathcal {M}}
$
be any cointertwining operator. Since $V$ is corational, we can assume $\mathcal {N}^1,\mathcal {N}^2$ are simple. Let $\mathcal {N}^1=\mathcal {M}^a,\mathcal {N}^2=\mathcal {M}^b$ for some $a,b\in I.$
Then
$
\mathcal{W}(z)\in \mathcal{I}_V
\binom{\mathcal {M}^a\quad \mathcal {M}^b}{\mathcal {M}}.
$
Therefore, there exist unique scalars
$
c_s^{a,b}\in\mathbb{C},
s=1,\ldots,N^{\mathcal {M}^a,\mathcal {M}^b}_\mathcal {M},
$
such that
\[
 \mathcal{W}(z)
=
\sum_{s=1}^{N^{\mathcal {M}^a,\mathcal {M}^b}_\mathcal {M}}
c^{a,b}_{s}\mathcal{W}^{a,b}_s(z).
\]

Now define
\[
\varphi_1=\sum_{a,b\in I}c_s^{a,b}\Id:\bigoplus_{s=1}^{N_\mathcal {M}^{\mathcal {M}^a,\mathcal {M}^b}}\mathcal {M}^a\rightarrow\mathcal {M}^a,~\varphi_2=\Id:\mathcal {M}^b\rightarrow\mathcal {M}^b,
\]
we have
\[
\mathcal {W}(z)=(\varphi_1\otimes \varphi_2)\circ \mathcal {W}_\mathcal {M}(z).
\]
This completes the proof.                                         $\hfill\Box$

From Proposition \ref{Prop4.3}, we immediately have the following Corollary.
\begin{Corollary}
The cotensor decomposition $\Delta_{\boxtimesc}$ is cocommutative.
\end{Corollary}

\begin{Lemma}\label{Lm5.4}
Let $V$ be a graded vertex operator coalgebra and $\mathcal {M}^i,\mathcal {M}^j,\mathcal {M}^k$ admissible $V$-comodules with finite-dimensional homogeneous subspaces. Then we have $$N^{\mathcal {M}^i,\mathcal {M}^j}_{\mathcal {M}^k}=N_{(\mathcal {M}^i)',(\mathcal {M}^j)'}^{(\mathcal {M}^k)'},$$
\end{Lemma}
\proof From Proposition \ref{Prop2.12} and \cite{W}, we know $V'$ is a rational vertex operator algebra and $(\mathcal {M}^i)',(\mathcal {M}^j)',(\mathcal {M}^k)'$ are admissible $V'$-modules. Let $\mathcal {W}(z)\in I_V\tbinom{\mathcal {M}^i~~\mathcal {M}^j}{\mathcal {M}^k}$. For any $m^{i*}\in(\mathcal {M}^i)',m^{j*}\in(\mathcal {M}^j)',m^{k}\in\mathcal {M}^k$, define$$\left(\mathcal {W}'(m^{i*},z)m^{j*},m^{k}\right)=\left(m^{i*}\otimes m^{j*},\mathcal {W}(z)m^{k}\right).$$
It is obvious that $\mathcal {W}'(\cdot,z)\in I_{V'}\tbinom{(\mathcal {M}^k)'}{(\mathcal {M}^i)'~~(\mathcal {M}^j)'}.$

Similarly, we have a linear map from $I_{V'}\tbinom{(\mathcal {M}^k)'}{(\mathcal {M}^i)'~~(\mathcal {M}^j)'}$ to $I_V\tbinom{\mathcal {M}^i~~\mathcal {M}^j}{\mathcal {M}^k}$, and above two procedure are inverse to each other.
                                $\hfill\Box$

\begin{Proposition}\label{Prop5.4}
The cotensor decomposition $\Delta_{\boxtimesc}$ is coassociative.
\end{Proposition}
\proof By linearity, we just need to prove
$$\Id\otimes\Delta_{\boxtimesc}\circ \Delta_{\boxtimesc}(\mathcal {M}^i)=\Delta_{\boxtimesc}\otimes\Id\circ \Delta_{\boxtimesc}(\mathcal {M}^i).$$

From Proposition \ref{Prop5.2}, the left hand side is equal to
$$\Id\otimes\Delta_{\boxtimesc}\circ \Delta_{\boxtimesc}(\mathcal {M}^i)=\sum_{a,b\in I}\sum_{c,d\in I}N^{\mathcal {M}^a,\mathcal {M}^b}_{\mathcal {M}^i}N^{\mathcal {M}^c,\mathcal {M}^d}_{\mathcal {M}^b}\mathcal {M}^a\boxtimes\mathcal {M}^c\boxtimes \mathcal {M}^d,$$
and the right hand side is equal to
$$\Delta_{\boxtimesc}\otimes\Id\circ \Delta_{\boxtimesc}(\mathcal {M}^i)=\sum_{b,d\in I}\sum_{a,c\in I}N^{\mathcal {M}^b,\mathcal {M}^d}_{\mathcal {M}^i}N^{\mathcal {M}^a,\mathcal {M}^c}_{\mathcal {M}^b}\mathcal {M}^a\boxtimes\mathcal {M}^c\boxtimes \mathcal {M}^d.$$
Hence, it is enough to show
$$\sum_{b\in I}N^{\mathcal {M}^a,\mathcal {M}^b}_{\mathcal {M}^i}N^{\mathcal {M}^c,\mathcal {M}^d}_{\mathcal {M}^b}=\sum_{b\in I}N^{\mathcal {M}^b,\mathcal {M}^d}_{\mathcal {M}^i}N^{\mathcal {M}^a,\mathcal {M}^c}_{\mathcal {M}^b}.$$
Using above lemma, this is equivalent to show
$$\sum_{b\in I}N^{(\mathcal {M}^i)'}_{(\mathcal {M}^a)',(\mathcal {M}^b)'}N^{(\mathcal {M}^b)'}_{(\mathcal {M}^c)',(\mathcal {M}^d)'}=\sum_{b\in I}N^{(\mathcal {M}^i)'}_{(\mathcal {M}^b)',(\mathcal {M}^d)'}N^{(\mathcal {M}^b)'}_{(\mathcal {M}^a)',(\mathcal {M}^c)'}.$$
This identity is obviously from Proposition \ref{Prop2.7}.                                          $\hfill\Box$

Now, we are ready to prove the following Theorem.
\begin{Theorem}
Recall the notation $\mathbb{K}[\mathcal{C}]=\mathbb{C}\otimes_{\mathbb{Z}}K[\mathcal{C}].$
For an admissible $V$-comodule $\mathcal {M}$, define $\epsilon_V(\mathcal {M})=\delta_{V,M},$ where $\delta$ is the Kronecker delta, extended $\epsilon_V$ to $\mathbb{K}[\mathcal{C}]$ linearly. Then $(\mathbb{K}[\mathcal{C}],\Delta_{\boxtimesc},\epsilon_V)$ is a cocommutative coassociative coalgebra.
\end{Theorem}
\proof From above discussions, we just need to prove the counit axioms. This is equivalent to show
\[
N^{V,\mathcal {M}^j}_{\mathcal {M}^i}=\delta_{i,j}.
\]

First, if $i=j$, the comodule structure map gives $N^{V,\mathcal {M}^i}_{\mathcal {M}^i}\geq1.$ On the other hand, from identity (\ref{LUM3.1}), we have $N^{V,\mathcal {M}^i}_{\mathcal {M}^i}=1.$

Second, if $j\neq i,$ for any cointertwining operator $\mathcal {W}(z)\in I_V\tbinom{V~~\mathcal {M}^j}{\mathcal {M}^i}$, from Lemma \ref{Lm5.4}, we have a intertwining operator $\mathcal {W}'(\cdot,z)\in I_{V'}\tbinom{(\mathcal {M}^i)'}{V'~~(\mathcal {M}^j)'}$. This is impossible. Hence we are done.

Now, it is trivial to verify that $(\mathbb{K}[\mathcal{C}],\Delta_{\boxtimesc},\epsilon_V)$ is a cocommutative coassociative coalgebra.                                         $\hfill\Box$

\section{Relationships to vertex operator algebras}
In this section, we assume $V$ is a simple and corational graded vertex operator coalgebra. Then $V'$ is a simple and rational vertex operator algebra. Let $\{\mathcal {M}^i\}_{i\in I}$ be the complete set of all simple admissible $V$-comodules and $\mathcal {C}$ the category of admissible $V$-comodules. From the construction of contragredient comodules and Proposition \ref{Prop2.12}, we can regard $\mathcal {C}^{op}$ as the category of admissible $V'$-modules. Let $K[\mathcal{C}]$(resp. $K[\mathcal{C}^{op}]$) be the Grothendieck group of $\mathcal{C}$(resp. $\mathcal {C}^{op}$) and set
\[
\mathbb{K}[\mathcal{C}]=\mathbb{C}\otimes_{\mathbb{Z}}K[\mathcal{C}],~~\mathbb{K}[\mathcal{C}^{op}]=\mathbb{C}\otimes_{\mathbb{Z}}K[\mathcal{C}^{op}].
\]

From section 5, we know $\mathbb{K}[\mathcal{C}]$ is a finite-dimensional cocommutative coassociative coalgebra. On the other hand, it is well-known that $\mathbb{K}[\mathcal{C}^{op}]$ is a finite-dimensional commutative associative algebra. Now we have the following results.
\begin{Proposition}
(i) $\mathbb{K}[\mathcal{C}^{op}]$ is the dual algebra of $\mathbb{K}[\mathcal{C}]$.

(ii) $\mathbb{K}[\mathcal{C}]$ is the dual coalgebra of $\mathbb{K}[\mathcal{C}^{op}]$.
\end{Proposition}
\proof (i) From Proposition \ref{Prop2.12}, we know $\mathbb{K}[\mathcal{C}^{op}]$ is the dual space of $\mathbb{K}[\mathcal{C}]$. Furthermore, $\{[(\mathcal {M}^i)']\}_{i\in I}$ is the dual basis of $\{[\mathcal {M}^i]\}_{i\in I}$. By linearity, it is enough to prove $$\left([(\mathcal {M}^i)']*[(\mathcal {M}^j)'],[\mathcal {M}^k]\right)=\left([(\mathcal {M}^i)']\otimes[(\mathcal {M}^j)'],\Delta_{\boxtimesc}([\mathcal {M}^k])\right).$$

By definition, the left hand side is equal to
$$\left(\sum_{s\in I}N^{(\mathcal {M}^s)'}_{(\mathcal {M}^i)',(\mathcal {M}^j)'}[(\mathcal {M}^s)'],[\mathcal {M}^k]\right)=N^{(\mathcal {M}^k)'}_{(\mathcal {M}^i)',(\mathcal {M}^j)'},$$
and the right hand side is equal to
$$\left([(\mathcal {M}^i)']\otimes[(\mathcal {M}^j)'],\sum_{s,t\in I}N_{\mathcal {M}^k}^{\mathcal {M}^s,\mathcal {M}^t}[\mathcal {M}^s]\otimes[\mathcal {M}^t]\right)=N_{\mathcal {M}^k}^{\mathcal {M}^i,\mathcal {M}^j}.$$
Now from Lemma \ref{Lm5.4}, we have $N^{(\mathcal {M}^k)'}_{(\mathcal {M}^i)',(\mathcal {M}^j)'}=N_{\mathcal {M}^k}^{\mathcal {M}^i,\mathcal {M}^j}.$ This proves (i).

The proof of (ii) is similar.                                                              $\hfill\Box$

\end{document}